\documentclass[11pt,fleqn]{article}

\usepackage{epsfig}
\usepackage{amsmath}
\usepackage{amssymb}
\usepackage{amscd}
\usepackage{latexsym}
\usepackage{tabularx}
\usepackage{a4wide}
\usepackage[usenames]{color}
\usepackage{enumerate}
\usepackage{subfigure}
\usepackage{url}

\small\normalsize

\newtheorem{theorem}{Theorem}
\newtheorem{corollary}[theorem]{Corollary}
\newtheorem{lemma}[theorem]{Lemma}

\newtheorem{observation}[theorem]{Observation}

\newcommand{\ch}{\mathit{ch}}

\newcommand{\MF}{\mathcal{F}}

\newcommand{\MH}{\mathcal{H}}

\newenvironment{proof}{\begin{trivlist}\item[]{\bf Proof.}\mbox{ \ }}%
        {\qquad\hspace*{\fill}$\Box$\end{trivlist}}

\newcommand{\qitee}[1]{\noindent\leavevmode\hangindent1.5\parindent%
        \noindent\hbox to1.5\parindent{#1\hss}\ignorespaces}

\title{\textbf{Distance-two coloring of sparse graphs}}

\author{\quad\\Zden\v{e}k
  Dvo\v{r}\'ak\,$^\ast$, \ Louis Esperet\,$^\dagger$\\[3mm]
  $^\ast$\,\emph{Computer Science Institute,
    Charles University, Prague, Czech Republic}\\[1mm] $^\dagger$\,\emph{CNRS
    -- Laboratoire G-SCOP, Grenoble, France}}

\date{ }

\begin{document}

\maketitle

{\renewcommand{\thefootnote}{\relax}

  \footnotetext{This work was partially supported by ANR Project
    HEREDIA, under grant \textsc{anr-10-jcjc-0204-01}, and by the
    project LL1201 (Complex Structures: Regularities in Combinatorics and Discrete Mathematics)
    of the Ministry of Education of Czech Republic.}}

\begin{abstract}
\noindent
Consider a graph $G = (V, E)$ and, for each vertex $v \in V$, a subset
$\Sigma(v)$ of neighbors of $v$. A $\Sigma$-coloring is a coloring of
the elements of $V$ so that vertices appearing together in some
$\Sigma(v)$ receive pairwise distinct colors. An obvious lower bound
for the minimum number of colors in such a coloring is the maximum
size of a set $\Sigma(v)$, denoted by $\rho(\Sigma)$. In this paper we
study graph classes $\MF$ for which there is a function $f$, such that
for any graph $G \in \MF$ and any $\Sigma$, there is a
$\Sigma$-coloring using at most $f(\rho(\Sigma))$ colors. It is proved
that if such a function exists for a class $\MF$, then $f$ can be
taken to be a linear function. It is also shown that such classes are
precisely the classes having bounded star chromatic number.  We also
investigate the list version and the clique version of this problem,
and relate the existence of functions bounding those parameters to the
recently introduced concepts of classes of bounded expansion and
nowhere-dense classes.

  \bigskip\noindent
  \emph{Keywords}: Distance-two coloring; star coloring; sparse graphs.
\end{abstract}

\section{Introduction}

Consider an undirected simple graph $G=(V,E)$ and suppose that for each vertex $v\in V$, we
are given a subset $\Sigma(v)\subseteq N(v)$, where $N(v)$ denotes the set of
vertices adjacent to $v$ in $G$. We
call such a collection a \emph{neighborhood system} for~$G$.

A \emph{$\Sigma$-coloring} is an assignment of colors to the elements
of $V$ so that any pair of vertices appearing together in
some~$\Sigma(v)$ must receive different colors. It can also be seen as
a strong coloring of the natural hypergraph defined by $\Sigma$ on the
vertex set $V$. Note that this coloring may not be
a proper coloring of $G$ (a condition that was required
in~\cite{AEH08}). It turns out that adding this condition
would not affect the results, so we prefer to use the simpler
definition here.  We let $\chi(\Sigma)$ denote the
minimum number of colors required for a $\Sigma$-coloring to
exist. 

When additionally each vertex~$v$ has its own list~$L(v)$ of colors
from which its color must be chosen, we talk about a
\emph{$\Sigma$-$L$-coloring}.  We define $\ch(\Sigma)$
as the minimum integer~$k$ such that for each assignment $L$ of
lists of at least~$k$ colors to vertices $v\in V$, there 
a $\Sigma$-$L$-coloring.

Given a graph $G=(V,E)$ and a neighborhood system $\Sigma$ for $G$, consider the
graph $G_{\Sigma}$ with vertex set $V$ and with the edge set defined as follows: two
vertices $u$ and $v$ are adjacent in $G_{\Sigma}$ if and only if $u$
and $v$ have a common neighbor $w$ such that $\{u,v\}\subseteq
\Sigma(w)$. Observe that $\chi(\Sigma)$ is precisely $\chi(G_\Sigma)$,
the chromatic number of $G_\Sigma$, and $\ch(\Sigma)$ is precisely
$\ch(G_\Sigma)$, the choice number of $G_\Sigma$.

We denote by $\rho(\Sigma)$ the quantity $\max_{v\in
  V}|\Sigma(v)|$. Since $G_\Sigma$ contains a clique of size $\rho(\Sigma)$, we clearly have $\rho(\Sigma) \le \chi(\Sigma) \le
\ch(\Sigma)$. In~\cite{AEH08}, it was proved that if $\MF_S$ is
the class of all graphs embeddable on a given surface $S$, then for
any $G \in \MF_S$ and any neighborhood system $\Sigma$ for $G$, $\ch(\Sigma)\le
\tfrac32 \rho(\Sigma)+o(\rho(\Sigma))$. As direct consequences of
this result, asymptotical versions of a conjecture of
Wegner~\cite{Weg77} and a conjecture of Borodin~\cite{Bor84} (see
also~\cite{JeTo95}) were obtained.

\medskip

A natural problem is to give a precise characterization of classes
$\MF$ for which there exists a function $f_\MF$ such that every $G \in \MF$ and
every neighborhood system $\Sigma$ for $G$ satisfies $\chi(\Sigma) \le
f_\MF(\rho(\Sigma) )$. If such a function $f_\MF$ exists, we say
that $\MF$ is \emph{$\sigma$-bounded}. Consider $C_\MF=\sup_{G\in \MF,
  \Sigma}\chi(\Sigma)/\rho(\Sigma)$, where the supremum is taken
over all graphs $G \in \MF$ and non-empty neighborhood systems $\Sigma$ for $G$. If
$C_\MF < \infty$, we say that $\MF$ is \emph{linearly
  $\sigma$-bounded}.

\medskip

As mentioned above, the class of graphs embeddable on a fixed surface
is linearly $\sigma$-bounded. It is not difficult to derive from
this result and the Graph Minor Theorem~\cite{RS03} that any proper
minor-closed class $\MF$ is also linearly
$\sigma$-bounded~\cite{AEH08}, but the constant $C_\MF$ obtained from
the second result is huge.

For a graph $G$, the \emph{1-subdivision} $G^*$ of $G$ is the graph
obtained from $G$ by subdividing every edge exactly once, i.e. by
replacing every edge by a path with 2 edges. For some integer $n\ge 4$,
consider $K_n^*$, the 1-subdivision of $K_n$, and for every vertex of
degree two in $K_n^*$, set $\Sigma(v)=N(v)$, and otherwise set
$\Sigma(v)=\emptyset$. A $\Sigma$-coloring is then
exactly a proper coloring of $K_n$, so $\chi(\Sigma) \ge n$, while
$\rho(\Sigma)=2$. Hence, the class containing the 1-subdivisions
of all complete graphs is not $\sigma$-bounded.

It follows that there exist classes that have bounded maximum average
degree (or degeneracy, or arboricity), but that are not
$\sigma$-bounded.

\medskip

A \emph{star coloring} of a graph $G$ is a proper coloring of the
vertices of $G$ so that every pair of color classes induces a forest
of stars. The \emph{star chromatic number} of $G$, denoted
$\chi_s(G)$, is the least number of colors in a star coloring of
$G$. We say that a class $\MF$ of graphs has \emph{bounded star chromatic
number} if the supremum of $\chi_s(G)$ for all $G \in \MF$ is finite.

\medskip

In Section~\ref{sec:1}, we will prove that a class $\MF$ is
$\sigma$-bounded if and only if it is linearly $\sigma$-bounded. In
order to do so, we will prove that both statements are equivalent to
the fact that $\MF$ has bounded star chromatic number. One of the
consequences of our main result is that there is a constant $c$ such
that every $K_t$-minor free graph $G$ and every neighborhood system $\Sigma$ for
$G$ satisfies $\chi(\Sigma)\le c t^4 (\log t)^2 \rho(\Sigma)$.

\smallskip

In Section~\ref{sec:choos}, we will extend this result to the list
version of the problem using results about arrangeability proved in
Section~\ref{sec:arr}. In Section~\ref{sec:clique} we will consider
the clique version of the problem, which has interesting connections
with the 2VC-dimension of hypergraphs. Finally, in
Section~\ref{sec:exp}, we will analyze the connections between the
extension of the problem to balls of higher radius and the recently
introduced notions of \emph{classes of bounded expansion} and
\emph{nowhere-dense classes}.

\section{Star Coloring}\label{sec:1}

In this section, we will prove the main result of this paper.

\begin{theorem}\label{th:1}
  Let $\MF$ be a class of graphs. The following four propositions are
  equivalent:

\smallskip
  \qitee{(i)} $\MF$ is linearly $\sigma$-bounded.

\smallskip
  \qitee{(ii)} $\MF$ is $\sigma$-bounded.

\smallskip
  \qitee{(iii)} The supremum of $\chi(\Sigma)$, over all $G \in \MF$
  and all neighborhood systems $\Sigma$ for $G$ with $\rho(\Sigma)=2$, is finite.

\smallskip
  \qitee{(iv)} $\MF$ has bounded star chromatic number.
\end{theorem}

The chain of implications (i) $\Rightarrow$ (ii) $\Rightarrow$ (iii) is
trivial. The implication (iii) $\Rightarrow$ (iv) can be derived from
Lemmas~\ref{lem:mindeg} and \ref{lem:subdiv} below, which follow from
the results of Dvo\v r\'ak~\cite{Dvo07}.

\begin{lemma}\label{lem:mindeg} {\bf \cite[Lemma 7]{Dvo07}}
There exists a function $g_2$ such that for every $k>0$,
every graph with minimum degree at least $g_2(k)$ 
contains as a subgraph the $1$-subdivision of a graph with
chromatic number $k$.
\end{lemma}

The \emph{acyclic chromatic number} of a graph $G$, denoted by
$\chi_a(G)$, is the least number of colors in a proper coloring of the
vertices of $G$ such that every pair of color classes induces a
forest. Since a star coloring is also an acyclic coloring, we have
$\chi_a(G) \le \chi_s(G)$. On the other hand, Albertson \emph{et al.}~\cite{AGKKR04}
proved that $\chi_s(G)\le \chi_a(G)(2\chi_a(G)-1)$ for any graph
$G$. The following claim is a simple combination of
Lemma~\ref{lem:mindeg} and the main result of~\cite{Dvo07}.

\begin{lemma}\label{lem:subdiv}
There exists a function $g$ such that for every $k>0$,
every graph with star chromatic number at least $g(k)$
contains as a subgraph the $1$-subdivision of a graph with
chromatic number $k$.
\end{lemma}
\begin{proof}
By Theorem 3 of \cite{Dvo07}, there exists a function $g_1$ such that
for every $k>0$ and every graph $G$ with acyclic chromatic number at
least $g_1(k)$, there exists a graph $H$ with chromatic number $k$
such that either $H$ or the $1$-subdivision of $H$ is a subgraph of
$G$.  Let $g_2$ be the function of Lemma~\ref{lem:mindeg}.  Let
$g(k)=2\bigl(g_1(\max(k,g_2(k)+1))\bigr)^2$.

Suppose that $G$ is a graph with star chromatic number at least
$g(k)$.  By Albertson {\it et al.}~\cite{AGKKR04}, $G$ has acyclic
chromatic number at least $\sqrt{g(k)/2}=g_1(\max(k,g_2(k)+1))$.
Hence, there exists a graph $H$ with chromatic number
$\max(k,g_2(k)+1)$ such that either $H$ or the $1$-subdivision of $H$
is a subgraph of $G$.  In the latter case, we choose a subgraph of $H$
with chromatic number $k$.  In the former case, since $H$ has
chromatic number at least $g_2(k)+1$, it contains a subgraph with
minimum degree at least $g_2(k)$.  Consequently, there exists a graph
$H'$ with chromatic number $k$ such that the $1$-subdivision of $H'$
is a subgraph of $H\subseteq G$.  In both cases, the $1$-subdivision
of a graph with chromatic number $k$ is a subgraph of $G$.
\end{proof}

\medskip

\noindent
\emph{Proof of (iii) $\Rightarrow$ (iv) of Theorem~\ref{th:1}.} Suppose that there exists a
constant $k$ such that for every $G\in \MF$ and every neighborhood system $\Sigma$
for $G$ with $\rho(\Sigma)=2$, we have $\chi(\Sigma)\le k$.  Let $g$
be the function of Lemma~\ref{lem:subdiv}.

Suppose that a graph $G\in\MF$ has star chromatic number at least
$g(k+1)$.  Then, there exists a graph $H$ with chromatic number $k+1$
such that $G$ contains the $1$-subdivision $H^*$ of $H$ as a subgraph.
For every vertex $v\in V(H^*)$ that has degree two in $H^*$, let
$\Sigma(v)=N_{H^*}(v)$.  For all other vertices of $G$, set
$\Sigma(v)=\emptyset$.  Note that a $\Sigma$-coloring of $G$ by $k$
colors induces a proper coloring of $H$, which is a contradiction
since $H$ has chromatic number $k+1$.  Therefore, the star chromatic
number of every graph in $\MF$ is less than $g(k+1)$.\hfill $\Box$

\medskip

The idea in the proof above will be used repeatedly in this paper: for
all graphs $G$ and $H$ such that $G$ contains the $1$-subdivision
$H^*$ of $H$ as a subgraph, there is a neighborhood system $\Sigma$ for $G$ with
$\rho(\Sigma)=2$ such that $G_\Sigma$ is isomorphic to $H$ together with a set
of isolated vertices.

\smallskip

Before we prove the implication (iv) $\Rightarrow$ (i), we need to
introduce a couple of definitions.  For a graph $G$, a \emph{colored
  in-orientation} of $G$ is a pair $(c,\overrightarrow{G})$, where
$\overrightarrow{G}$ is an orientation of $G$ and $c$ is a proper
coloring of $G$ such that for every 2-colored path on 3 vertices in
$G$, the edges of the path are oriented toward the middle
vertex. Albertson {\it et al.}~\cite{AGKKR04}, and independently
Ne\v{s}et\v{r}il and Ossona de Mendez~\cite{NO03}, observed the
following:

\begin{observation}\label{obs:star}
A proper coloring $c$ of a graph $G$ is a star coloring if and
only if there exists an orientation $\overrightarrow{G}$ of $G$
such that $(c,\overrightarrow{G})$ is a colored in-orientation of $G$.
Consequently, $G$ has star chromatic number at most $k$ if and only if $G$
has a colored in-orientation using $k$ colors.
\end{observation}

We now prove the implication (iv) $\Rightarrow$ (i) in
Theorem~\ref{th:1}, that is, if $\MF$ has bounded star chromatic
number, then $\MF$ is linearly $\sigma$-bounded. This is a direct
consequence of Lemma~\ref{lem:1} below.

\begin{lemma}\label{lem:1}
  Every graph $G$ and every neighborhood system $\Sigma$ for $G$ satisfies $\chi(\Sigma) \le
  \chi_s(G)^2 \rho(\Sigma)$.
\end{lemma}
\begin{proof}
Take a graph $G$ and a neighborhood system $\Sigma$ for $G$. Consider a colored
in-orientation $(c_1,\overrightarrow{G})$ of $G$ using at most $k=\chi_s(G)$
colors. Note that by definition of a colored in-orientation,
every vertex has out-degree at most $k$ in $\overrightarrow{G}$. Let $G_2$ be the graph with
vertex set $V(G_2)=V(G)$ and edge set $E(G_2)$ defined as follows: $uv
\in E(G_2)$ if and only if there exists $w$ such that
$\overrightarrow{uw}, \overrightarrow{vw} \in E(\overrightarrow{G})$
and $u,v\in \Sigma(w)$.  For every $v\in V(G_2)$, the degree of $v$ in
$G_2$ is at most the sum of $|\Sigma(w)|-1$ over all vertices $w \in
V(G)$ satisfying $\overrightarrow{vw} \in E(\overrightarrow{G})$.
Since every vertex has out-degree at most $k$ in $\overrightarrow{G}$,
we conclude that the maximum degree of $G_2$ is at most
$k(\rho(\Sigma)-1)$.  Hence, there is a proper coloring $c_2$ of
$G_2$ using at most $k(\rho(\Sigma)-1)+1 \le k \rho(\Sigma)$
colors.

For any vertex $v\in V(G)$, define $c(v)=(c_1(v),c_2(v))$.
This coloring uses at most $k^2 \rho(\Sigma)$
colors. We now prove that $c$ is a $\Sigma$-coloring of $G$. Take a
vertex $w$ of $G$ and two neighbors $u,v$ of $w$ with $u,v \in
\Sigma(w)$. If $c_1(u)=c_1(v)$, we know by the definition of an
colored in-orientation that $\overrightarrow{uw}, \overrightarrow{vw} \in
E(\overrightarrow{G})$. In this case, by the definition of $G_2$ and
$c_2$, we have $c_2(u)\ne c_2(v)$. This shows that $\chi(\Sigma)\le
k^2 \rho(\Sigma)$ and concludes the proof of Lemma~\ref{lem:1}.
\end{proof}

\medskip

It follows from~\cite{NO03} that graphs with no $K_t$-minor have star
chromatic number $O(t^2 \log t)$. Hence, we obtain the following
corollary of Lemma~\ref{lem:1}:

\begin{corollary}
There is a constant $c$ such that every $K_t$-minor free graph $G$
and every neighborhood system $\Sigma$ for $G$ satisfies $\chi(\Sigma)\le c t^4 (\log t)^2
\rho(\Sigma)$.
\end{corollary}

\section{Arrangeability}\label{sec:arr}

Recall that the graph $G_{\Sigma}$ with vertex set $V$ was defined as
follows: two vertices $u$ and $v$ are adjacent in $G_{\Sigma}$ if and
only if $u$ and $v$ have a common neighbor $w$ such that
$\{u,v\}\subseteq \Sigma(w)$. We define $\text{mad}(\Sigma)$ as the
maximum average degree over all subgraphs of $G_{\Sigma}$. Note that
we have $\chi(\Sigma)\le \ch(\Sigma)\le \lfloor
\text{mad}(\Sigma)\rfloor+1$.

\smallskip

The \emph{arrangeability} of a graph $G$ is the least $k$ such that
there exists a linear ordering $\prec$ of the vertices of $G$
satisfying the following: for each vertex $v$, there are at most $k$
vertices smaller than $v$ that have a common neighbor $u$ with $v$ such that $v
\prec u$. It can be observed that any $k$-arrangeable graph is
$(k+1)$-degenerate and has acyclic chromatic number at most
$2k+2$, see~\cite{Dvo07}.

\smallskip

A class $\MF$ has \emph{bounded arrangeability} if there exists a constant $k$ such that every
graph in $\MF$ has arrangeability at most $k$.
We now prove the following result
relating the maximum average degree of neighborhood system $\Sigma$ for graphs in a
class $\MF$ and the arrangeability.

\begin{theorem}\label{th:arr}
  Let $\MF$ be a class of graphs. The following
  statements are equivalent:

\smallskip \qitee{(i)} There is a constant $C_\MF$ such that every
  $G \in \MF$ and every neighborhood system $\Sigma$ for $G$ satisfies $\text{mad}(\Sigma)
  \le C_\MF \cdot \rho(\Sigma)$.

\smallskip
  \qitee{(ii)} There is a function $f_\MF$ such that every
  $G \in \MF$ and every neighborhood system $\Sigma$ for $G$ satisfies $\text{mad}(\Sigma) \le
  f_\MF( \rho(\Sigma))$.

\smallskip \qitee{(iii)} The supremum of $\text{mad}(\Sigma)$ over all
  graphs in $\MF$ and all neighborhood systems $\Sigma$ for $G$ with
  $\rho(\Sigma)=2$ is finite.

\smallskip
  \qitee{(iv)} $\MF$ has bounded arrangeability.

\end{theorem}

The chain of implications (i) $\Rightarrow$ (ii) $\Rightarrow$ (iii)
is trivial. We now prove the implications (iii) $\Rightarrow$ (iv) and
(iv) $\Rightarrow$ (i).

\medskip

\noindent
\emph{Proof of (iii) $\Rightarrow$ (iv).}  Assume first that the
maximum average degree of the graphs in $\MF$ is not bounded. Then we
can find subgraphs of graphs in $\MF$ with arbitrarily large minimum
degree. It follows from Lemma~\ref{lem:mindeg} that we can find
1-subdivisions of graphs with arbitrarily large chromatic number (and
hence, minimum degree) as subgraphs of graphs of $\MF$. The
corresponding neighborhood systems $\Sigma$ with $\rho(\Sigma)=2$
witness that $\text{mad}(\Sigma)$ is unbounded.

We may now assume that there exists a constant $d_0$ such that every
graph of $\MF$ has maximum average degree at most $d_0$.  By Dvo\v
r\'ak~\cite[Theorem 9]{Dvo07}, every graph $G$ with maximum average
degree at most $d$ and arrangeability more than $4d^2(4d+5)$ contains
the 1-subdivision of a graph $H$ with minimum degree at least $d$ as a
subgraph. Hence, if $\MF$ contains graphs with arbitrarily large
arrangeability, then we can again find a neighborhood system $\Sigma$
with $\rho(\Sigma)=2$ and $\text{mad}(\Sigma)\ge d$ for every $d\ge
d_0$, contradicting (iii).\hfill $\Box$

\medskip

\noindent
\emph{Proof of (iv) $\Rightarrow$ (i).}  Take the linear ordering
$\prec$ of the vertices of $G\in \MF$ witnessing that $G$ has
arrangeability at most $k$.  Note that each vertex of $G$
has at most $k+1$ neighbors preceding it according to $\prec$,
as otherwise the latest such neighbor would contradict
the $k$-arrangeability of $G$.

Consider a subgraph $H$ of $G_\Sigma$, and let $v$ be its last vertex
in the ordering.  We want to find an upper bound on the number of
neighbors $u$ of $v$ in $H$. For each such vertex $u$ there is a
vertex $w$ in $G$ with $u,v\in \Sigma(w)$. Since $vwu$ is a path in
$G$ and $u \prec v$, by the definition of arrangeability there are at
most $k$ choices of $u$ with $v \prec w$. Moreover, there are only
$k+1$ choices of $w$ with $w \prec v$ by the previous paragraph, and
for each choice of $w$ there are at most $|\Sigma(w)|\le
\rho(\Sigma)$ choices for $u$. Therefore, $v$ has at most
$k+(k+1)\rho(\Sigma) \le (2k+1)\rho(\Sigma)$ neighbors in $H$. It
follows that $G_\Sigma$ is $((2k+1)\rho(\Sigma))$-degenerate, and
thus $\text{mad}(\Sigma) \le (4k+2)\rho(\Sigma)$.\hfill $\Box$

\section{Choosability}\label{sec:choos}

In this section, we consider the list variant of $\Sigma$-coloring.
The \emph{star choice number} of a graph $G$, denoted $\ch_s(G)$, is the
minimum $k$ such that if every vertex of $G$ is given a list $L(v)$ of
size at least $k$, $G$ has a star coloring in which every vertex is
assigned a color from its list.

\begin{theorem}\label{th:2}
  Let $\MF$ be a class of graphs. The following propositions are
  equivalent:

\smallskip
  \qitee{(i)} $\MF$ has bounded star choice number.

  \smallskip

  \qitee{(ii)} $\MF$ has bounded arrangeability.

  \smallskip
  
  \qitee{(iii)} There is a constant $C_\MF$ such that every
  $G \in \MF$ and every neighborhood system $\Sigma$ for $G$ satisfies $\ch(\Sigma) \le
  C_\MF \cdot \rho(\Sigma)$.

  \smallskip

  \qitee{(iv)} There is a function $f_\MF$ such that every
  $G \in \MF$ and every neighborhood system $\Sigma$ for $G$ satisfies $\ch(\Sigma) \le
  f_\MF( \rho(\Sigma))$.

  \smallskip

  \qitee{(v)} The supremum of $\ch(\Sigma)$, over all $G \in \MF$
  and all neighborhood systems $\Sigma$ for $G$ with $\rho(\Sigma)=2$, is finite.
\end{theorem}

We use the following result of Kang~\cite{Kang12}.  A
\emph{$t$-improper coloring} of a graph $H$ is an assignment of colors
to its vertices such that each color induces a subgraph of maximum
degree at most $t$.

\begin{theorem}\label{thm:kang}{\bf \cite[Theorem 6]{Kang12}}
For every $k,t\ge 0$, there exists $D>0$ with the following property.
For every graph $H$ of minimum degree at least $D$, there exists
an assignment $L$ of lists of size $k$ to the vertices of $H$
such that $H$ has no $t$-improper coloring from $L$.
\end{theorem}

Let us first prove the equivalence of (i) and (ii), which is
interesting in its own right.

\begin{lemma}\label{lem:stararr}
  A class of graphs $\MF$ has bounded star choice number if and only if
  it has bounded arrangeability.
\end{lemma}
\begin{proof}
  Consider a graph $G\in \MF$.  First, we show that if $G$ has
  arrangeability at most $k$, then its star choice number is at most
  $(k+2)^2$.  Let $\prec$ be an ordering of the vertices of $G$
  witnessing its arrangeability.  Let $L$ be any assignment of lists
  of length $(k+2)^2$ to the vertices of $G$.  For each $v\in V(G)$,
  let $P(v)$ denote the set of vertices $u\prec v$ such that either
  $u$ is adjacent to $v$ or there exists a path $uwv$ with $u\prec w$.
  Recall that each vertex of $G$ has at most $k+1$ neighbors preceding
  it in $\prec$, thus $P(v)$ contains at most $k+1$ neighbors of $v$,
  and at most $(k+1)^2$ vertices $u$ such that there exists a path
  $uwv$ with $u\prec w\prec v$.  By arrangeability, $P(v)$ contains at
  most $k$ vertices $u$ such that there exists a path $uwv$ with
  $u\prec v\prec w$.  Consequently, $|P(v)|\le
  (k+1)^2+(k+1)+k<|L(v)|-1$.

  Therefore, we can $L$-color $G$ greedily so that each $v\in V(G)$
  has a color $c(v)$ different from the colors of the elements of
  $P(v)$.  Clearly, $c$ is a proper $L$-coloring of $G$.  Consider a
  path $P\subseteq G$ on four vertices.  Let $u$ be the earliest
  vertex of $P$ according to $\prec$.  Note that $P$ contains a path
  on three vertices starting with $u$.  Let $uwv$ be such a path.  We
  have $u\prec v$ and $u\prec w$, and thus $u\in P(v)$.  Consequently,
  $u$, $v$ and $w$ have distinct colors, and $P$ is not bichromatic.
  It follows that $c$ is a star coloring of $G$.  Since the choice of
  the lists $L$ was arbitrary, the star choice number of $G$ is at
  most $(k+2)^2$.

  Suppose now that $\MF$ has star choice number at most $k$ and let
  $D$ be the constant from Theorem~\ref{thm:kang} for $t=k$. As shown
  in the proof of the implication (iii) $\Rightarrow$ (iv) in
  Theorem~\ref{th:arr}, in order to prove that $\MF$ has bounded
  arrangeability, it suffices to prove that for every graph $H$ of
  minimum degree $D$, the $1$-subdivision of $H$ does not appear as a
  subgraph of any $G\in \MF$.  Suppose on the contrary that there
  exists a graph $H$ of minimum degree $D$ and $G\in \MF$ containing
  the $1$-subdivision $H^*$ of $H$ as a subgraph.  Let $L$ be the list
  assignment for $H$ from Theorem~\ref{thm:kang} for $t=k$.  Let $L^*$
  be the list assignment for $H^*$ matching $L$ on the vertices of $H$
  and assigning to all vertices of $H^*$ of degree two the list
  $\{1,\ldots, k\}$.  Since the star choice number of $G$ is at most
  $k$, there exists a star coloring $c$ of $H^*$ from $L^*$. By
  Observation~\ref{obs:star}, there is an orientation
  $\overrightarrow{H^*}$ of $H^*$ such that $(c,\overrightarrow{H^*})$
  is a colored in-orientation of $H^*$.  Since all vertices of $H^*$
  of degree two have the same list, the maximum out-degree of
  $\overrightarrow{H^*}$ is at most $k$.  Thus, each vertex $v\in
  V(H)$ has at most $k$ neighbors $u\in V(H)$ such that the vertex $w$
  subdividing the edge $uv$ in $H^*$ satisfies $\overrightarrow{vw},
  \overrightarrow{uw}\in \overrightarrow{H^*}$.  By the definition of
  a colored in-orientation, each $v\in V(H)$ has at most $k$ neighbors
  $u\in V(H)$ such that $c(u)=c(v)$.  Consequently, $c$ induces a
  $k$-improper coloring of $H$ from the lists $L$, contradicting
  Theorem~\ref{thm:kang}.
\end{proof}

\medskip

For the implication (ii) $\Rightarrow$ (iii), note that by
Theorem~\ref{th:arr}, there exists a constant $C'_\MF$ such that for
each $G\in\MF$ and each neighborhood system $\Sigma$ for $G$, the
graph $G_{\Sigma}$ has maximum average degree at most
$C'_\MF\cdot\rho(\Sigma)$.  Hence, $G_{\Sigma}$ has choice number at
most $C'_\MF\cdot\rho(\Sigma)+1\le(C'_\MF+1)\cdot\rho(\Sigma)$.
Therefore, (iii) holds with $C_\MF=C'_\MF+1$.

The implications (iii) $\Rightarrow$ (iv) $\Rightarrow$ (v) are trivial.

For the implication (v) $\Rightarrow$ (ii), let $k$ satisfy
$\ch(\Sigma)\le k$ for all $G \in \MF$ and all neighborhood systems $\Sigma$ for
$G$ with $\rho(\Sigma)=2$.  Let $D$ be the constant of
Theorem~\ref{thm:kang} with $t=0$. Then for all $G \in \MF$ and all
neighborhood systems $\Sigma$ for $G$ with $\rho(\Sigma)=2$ we have that
$G_\Sigma$ is $D$-degenerate and so $\text{mad}(\Sigma)\le 2D$. By
Theorem~\ref{th:arr}, it follows that $\MF$ has bounded
arrangeability. This concludes the proof of Theorem~\ref{th:2}.

\medskip

It is interesting to note that in general for a class $\MF$, it is not
equivalent to be $\sigma$-bounded and to have $\ch(\Sigma)$ bounded by
a function of $\rho(\Sigma)$. To prove this, it is enough by
Theorems~\ref{th:1} and~\ref{th:2} to construct a family of graphs
with bounded star chromatic number and unbounded arrangeability.  For
$n\ge 3$, let $H_n$ be the 1-subdivision of the complete bipartite
graph $K_{n,n}$. Note that every $H_n$ has star chromatic number at
most three. On the other hand, $H_n$ is a 1-subdivision of a graph
with minimum degree $n$, and thus its arrangeability is at least
$(n-1)/2$ (see the beginning of Section 2 in \cite{Dvo07} for an
argument showing this). Thus, the class $\{H_n\}_{n\ge 3}$ has bounded
star chromatic number and unbounded arrangeability.



\section{Cliques}\label{sec:clique}

Given a graph $G=(V,E)$ and a neighborhood system $\Sigma$ for $G$, a
\emph{$\Sigma$-clique} is a set $C$ of elements of $V$ such for that
any pair $u,v \in C$ there is a vertex $w \in V$ such that $u,v \in
\Sigma(w)$. The maximum size of a $\Sigma$-clique is denoted by
$\omega(\Sigma)$. Note that $\omega(\Sigma)$ is precisely
$\omega(G_\Sigma)$, the clique number of $G_\Sigma$, so we trivially
have $\rho(\Sigma) \le \omega(\Sigma) \le \chi(\Sigma)$.

In the case where $\Sigma(v)=N(v)$ for every vertex $v$ of $G$, the
parameter $\omega(\Sigma)$ has been studied for several graph
classes, such as planar graphs~\cite{CvdH,HS93} and
line-graphs~\cite{CGTT90,FSGT90}.

As previously, it is natural to investigate classes $\MF$ for which
there is a function $f$, such that for any $G \in \MF$ and any
neighborhood system $\Sigma$ for $G$, we have $\omega(\Sigma) \le
f(\rho(\Sigma))$.

Let $\MF$ be a class of graphs that contains as subgraphs of graphs
in $\MF$ the 1-subdivisions of arbitrarily large cliques. Then $\omega$
cannot be bounded by a function of $\rho$ in $\MF$. In the
following, we will show that the converse also holds.

\begin{theorem}\label{th:3}
  Let $\MF$ be a class of graphs, and denote by $\widetilde{\MF}$ the
  class containing all subgraphs of graphs in $\MF$. The following
  statements are equivalent:

  \smallskip \qitee{(i)} The supremum of $\omega(\Sigma)$ over all
  graphs in $\MF$ and all neighborhood systems $\Sigma$ for $G$ with
  $\rho(\Sigma)=2$ is finite.

\smallskip \qitee{(ii)} There is a constant $C_\MF$ such that
  every $G \in \MF$ and every neighborhood system $\Sigma$ for $G$ satisfies $\omega(\Sigma)
  \le C_\MF \cdot \rho(\Sigma)$.

\smallskip
  \qitee{(iii)} There is a function $f_\MF$ such that every
  $G \in \MF$ and every neighborhood system $\Sigma$ for $G$ satisfies $\omega(\Sigma) \le
  f_\MF( \rho(\Sigma))$.

\smallskip
  \qitee{(iv)} $\widetilde{\MF}$ does not contain 1-subdivisions of
  arbitrarily large cliques.
\end{theorem}

The implication (ii) $\Rightarrow$ (iii) is trivial, and (iii)
$\Rightarrow$ (iv) corresponds to the remark above. We now prove the
implications (iv) $\Rightarrow$ (i) and (i) $\Rightarrow$ (ii).

\medskip

\noindent
\emph{Proof of (iv) $\Rightarrow$ (i).}  Fix some integer $n$. If (i)
does not hold, then there is a graph $G\in \MF$ and a neighborhood system $\Sigma$
for $G$ with $\rho(\Sigma)=2$, such that there is a
$\Sigma$-clique $C$ of size at least $3n$. Consider the graph $H$ with
vertex set $C$, in which two vertices $u,v \in C$ are adjacent if and
only if there exists a vertex $w \not\in C$ with
$\Sigma(w)=\{u,v\}$. $H$ has at least ${3n \choose 2}-3n$ edges so by
Tur\'an's theorem $H$ contains a clique of size $(3n)^2/(2\times
3n+3n)=n$. It follows that $G$ contains (as a subgraph) the
1-subdivision of a clique of size $n$.\hfill $\Box$

\medskip

\noindent
\emph{Proof of (i) $\Rightarrow$ (ii).}  A $\Sigma$-clique can be
seen as a hypergraph in which every pair of vertices is contained in a
hyperedge (i.e. a hypergraph with strong stability number 1). We call
such a hypergraph a \emph{full} hypergraph. Given a hypergraph $\MH$,
a \emph{subhypergraph} of $\MH$ on a set $X \subseteq V(\MH)$ is
a hypergraph with the vertex set $X$ whose hyperedge set is a subset of
$\{X \cap e\,|\, e \in E(\MH)\}$. The \emph{rank} $r(\MH)$ of a hypergraph $\MH$
is the maximum size of a hyperedge of $\MH$.

\smallskip

To prove (i) $\Rightarrow$ (ii), it is enough to prove
the following:

\smallskip

\noindent ($*$) For every $r,n\ge 2$, every full hypergraph
of rank at most $r$ having at least $4rn^2+2$ vertices contains a full
subhypergraph with rank two on a set of at least $n$ vertices.

\smallskip 

Consider a full hypergraph $\MH$ of rank at most $r$, having $N\ge
4rn^2+2$ vertices. For each pair $u,v \in V(\MH)$ choose a hyperedge
$e_{u,v}$ containing $u$ and $v$. Given a subset $X \subseteq V(\MH)$,
we say that two vertices $u,v \in V(\MH)$ form \emph{a bad $X$-pair}
if $u,v\in X$ and $|X \cap e_{u,v}|>2$.

We now consider a subset $X \subset V(\MH)$ taken uniformly at random
among all subsets of $V(\MH)$ of size $2n$. Observe that choosing $X$
so that two given vertices $u$ and $v$ form a bad $X$-pair is the same
as choosing a vertex $w \in e_{u,v}$ distinct from $u$ and $v$, and
then $2n-3$ vertices in $V(\MH)$ distinct from $u$, $v$, and $w$.  Let
us remark that if $|X \cap e_{u,v}|>3$, then there are several
possible choices of $w$.  It follows that the probability that $u$ and
$v$ form a bad $X$-pair is at most $$p=\frac{(r-2){N-3 \choose
    2n-3}}{{N \choose 2n}}<\frac{8rn^3}{N(N-1)(N-2)}.$$ Therefore, the
expectation of the number of bad $X$-pairs is at most ${N\choose
  2}p<\tfrac{4rn^3}{N-2}$.

\smallskip

It follows that there exists a set $X \subset V(\MH)$ of size $2n$,
such that less than $\tfrac{4rn^3}{N-2}$ pairs of vertices of $V(\MH)$
are bad $X$-pairs. Since $N\ge 4rn^2+2$, less than $n$ pairs of
vertices from $X$ are bad $X$-pairs. Consider a graph $H$ with vertex
set $X$, and edges $uv$ if and only if $u,v$ do not form a bad
$X$-pair. This graph has $2n$ vertices and at most $n$ non-edges, so
by Tur\'an's theorem it contains a clique of size
$(2n)^2/(2n+2n)=n$. This clique corresponds to a subset $Y\subseteq
X$ of size $n$, such that no pair of vertices from $Y$ is a bad $Y$-pair. Hence,
there exists a full subhypergraph of $\MH$ of rank two on $Y$,
which concludes the proof.  \hfill $\Box$

\medskip

It is interesting to note that ($*$) can also be derived directly from
a result of Ding, Seymour, and Winkler~\cite{DSW94}. For the sake of
completeness, we include this alternative proof.

The \emph{strong stability number} $\alpha(\MH)$ of a hypergraph $\MH$
is the maximum size of a set $S\in V(\MH)$ such that no hyperedge of
$\MH$ contains more than one vertex of $S$. A \emph{fractional
  covering} of $\MH$ is a map $\phi : E(\MH) \rightarrow [0,1]$ such
that for any vertex $v$, $\sum_{e \ni v} \phi(e) \ge 1$. The
\emph{fractional covering number} $\kappa^*(\MH)$ of $\MH$ is the
infimum of $\sum_{e \in E(\MH)} \phi(e)$ over all fractional coverings
of $\MH$. Let $\lambda(\MH)$ denote the maximum size of a set $S
\subseteq V(\MH)$ such that for all distinct $u,v \in S$, there exists an
hyperedge $e_{u,v} \in E(\MH)$ such that $e_{u,v} \cap
S=\{u,v\}$. Ding, Seymour, and Winkler proved the following (see the
end of the proof of (5.7) in~\cite{DSW94} where the statement is in
the dual setting).

\begin{lemma}[\cite{DSW94}]\label{lem:dsw} For any hypergraph $\MH$,
  $\kappa^*(\MH)\le \tfrac{27}8\, {{\lambda(\MH)+\alpha(\MH)} \choose
    {\alpha(\MH)}}^2$.
\end{lemma}

\begin{corollary}
For any $r\ge 2$ and $n \ge 1$, every full hypergraph of rank at most
$r$ having at least $\tfrac{27}8\, r(n+1)^2$ vertices contains a full
subhypergraph with rank at most two having at least $n$ vertices.
\end{corollary}

\begin{proof}
Consider a full hypergraph $\MH$ of rank at most $r$, having $N\ge
\tfrac{27}8\, r(n+1)^2$ vertices. Observe that $\kappa^*(\MH)\ge N/r$,
so Lemma~\ref{lem:dsw} with $\alpha(\MH)=1$ gives $\tfrac{27}8
(\lambda(\MH)+1)^2\ge N/r$. It follows that $\lambda(\MH)\ge n$.  For any set $S$ witnessing this
inequality, there exists a full subhypergraph of $\MH$ on $S$
with rank two, which concludes the proof.
\end{proof}

A class that is $\sigma$-bounded has the property that
$\omega(\Sigma)$ is bounded by a function of $\rho(\Sigma)$, but
the converse does not hold in general. Take a class $\MF$ of
triangle-free graphs with arbitrarily large chromatic numbers, and
consider the class $\MF^*$ consisting of the 1-subdivisions of the
graphs in $\MF$. The class $\MF^*$ does not contain any graph
containing the 1-subdivision of a clique of size more than two. On the
other hand, $\MF^*$ is trivially not $\sigma$-bounded.

\section{Bounded expansion}\label{sec:exp}

The notion of a \emph{class of bounded expansion} was introduced by
Ne\v set\v ril and Ossona de Mendez
in~\cite{NO06,NO08a,NO08b,NO08c}. Examples of such classes include
minor-closed classes, topological minor-closed classes, classes
locally excluding a minor, and classes of graphs that can be drawn in
the plane with a bounded number of crossings per edge.  It was
recently proved that first-order properties can be decided in linear
time in classes of bounded expansion~\cite{DKR10}. Additionally, such
classes have deep connections with the existence of finite
homomorphism dualities~\cite{NO08c}.

For more about this topic, the reader is referred
to the survey book of Ne\v set\v ril and Ossona de
Mendez~\cite{NO}.

\smallskip

A graph $H$ is a \emph{shallow topological minor of $G$ at depth $d$}
if $G$ contains a $(\le\!2d)$-subdivision of $H$ (i.e. a graph obtained
from $H$ by subdividing each edge at most $2d$ times) as a
subgraph. For any $d\ge 0$, let $\widetilde{\nabla}_d(G)$ be defined
as the maximum of $|E(H)|/|V(H)|$, over all shallow topological minors
$H$ of $G$ at depth $d$. Note that
$\widetilde{\nabla}_0(G)=\text{mad}(G)/2$ and the function $d \mapsto
\widetilde{\nabla}_d(G)$ is monotone. If there is a function $f$
such that for every $d\ge 0$, every graph $G \in \MF$
satisfies $\widetilde{\nabla}_d(G) \le f(d)$, then $\MF$ is said to have
\emph{bounded expansion}.

\medskip

Let $G$ be a simple undirected graph.  For an integer $d\ge 1$ and a
vertex $v\in V(G)$, let $N^d(v)$ denote the set of vertices at
distance at most $d$ from $v$ in $G$, excluding $v$ itself.  We say
that $\Sigma$ is a \emph{$d$-neighborhood system} if $\Sigma(v)
\subseteq N^d(v)$ for all $v\in V(G)$.  A class of graphs $\MF$ is
said to be \emph{$\sigma$-bounded at depth $d$} if there exists a
function $f$ such that for any $G \in \MF$ and any $d$-neighborhood
system $\Sigma$ for $G$, we have $\chi(\Sigma)\le
f(\rho(\Sigma))$. Note that a class is $\sigma$-bounded at depth 1
precisely if it is $\sigma$-bounded.

\smallskip
\noindent
The following two propositions are equivalent.

\smallskip
  \qitee{(i)} There exists a constant $c$, such that every $G\in \MF$ and
  every $d$-neighborhood system $\Sigma$ for $G$ with $\rho(\Sigma)=2$ satisfies
  $\chi(\Sigma)\le c$.

\smallskip
  \qitee{(ii)} $\MF$ is $\sigma$-bounded at depth $d$.

\medskip

\emph{Proof of (i) $\Rightarrow$ (ii).} Assume that for any $G \in
\MF$ and any $d$-neighborhood system $\Sigma$ for $G$ with $\rho(\Sigma)=2$, we
have $\chi(\Sigma)\le k$.   Let $\ell={{\rho(\Sigma)} \choose 2}$.
For each vertex $v$ of $G$, let $\ell'={{|\Sigma(v)|} \choose 2}$, let $\Sigma_1(v),\ldots,\Sigma_{\ell'}(v)$
be all pairs of elements of $\Sigma(v)$, and let
$\Sigma_{\ell'+1}(v)=\ldots=\Sigma_{\ell}(v)=\emptyset$. For
$1\le i \le \ell$, let $c_i$ be a $\Sigma_i$-coloring using
at most $k$ colors. Then the coloring $c$ defined by
$c(v)=(c_1(v),\ldots,c_\ell(v))$ is a $\Sigma$-coloring using at
most $k^{{\rho(\Sigma)} \choose 2}$ colors.\hfill $\Box$

\medskip

We only use the trivial direction (ii) $\Rightarrow$ (i) to prove the
following result.

\begin{lemma}\label{lem:be1}
  If a class $\MF$ of graphs is $\sigma$-bounded at depth $d$ for every
  integer $d\ge 1$, then $\MF$ has bounded expansion.
\end{lemma}
\begin{proof}
Assume that there is a function $f$ such that for any integer $d\ge
1$, any graph $G\in \MF$ and any $d$-neighborhood system $\Sigma$ for
$G$ with $\rho(\Sigma)=2$, we have $\chi(\Sigma) \le
f(d)$.  Let $g_2$ be the function of Lemma~\ref{lem:mindeg}.

\smallskip

 We will prove that for any $d\ge 1$ and any $G\in \MF$,
 $\widetilde{\nabla}_d(G) < g_2(f(2d+1)+1)=a$. Assume for the sake of
 contradiction that it is not the case. Then there exists a graph $G\in\MF$
 and an integer $d \ge 1$ such that $\widetilde{\nabla}_d(G) \ge a$. It
 follows that $G$ has a subgraph that is a $(\le\!2d)$-subdivision of a
 graph $H$ of minimum degree at least $a$.
 Note that $H$ has a subgraph that is the 1-subdivision
 of a graph $H'$ with $\chi(H')=f(2d+1)+1$. Hence, $G$ has a subgraph
 that is a $(\le\!4d+1)$-subdivision of $H'$. Let $u_1,\ldots,u_t$ be the
 images of the vertices $v_1,\ldots,v_t$ of $H'$ in $G$. For any edge
 $v_iv_j$ in $H'$, select a vertex $u_{i,j}$ in $G$ lying in the middle
 of the path between $u_i$ and $u_j$ corresponding to the image of the
 edge $v_iv_j$ in $G$. Note that $u_{i,j}$ is at distance at most $2d+1$
 from $u_i$ and $u_j$, and all $u_{i,j}$ are pairwise distinct. Set
 $\Sigma(u_{i,j})=\{u_i,u_j\}$ for all $i,j$, and
 $\Sigma(v)=\emptyset$ for all the other vertices. Observe that
 $\Sigma$ is a $(2d+1)$-neighborhood system, $\rho(\Sigma)=2$, and yet
 $\chi(\Sigma) > f(2d+1)$, which contradicts the hypothesis.
\end{proof}

\medskip

Unfortunately, it turns out that the converse of Lemma~\ref{lem:be1} is not true. Consider
the graph $S_n$ consisting of a star with ${n \choose 2}+n$ leaves
$\{v_{i,j}\,|\,1 \le i < j \le n\}\cup \{v_i\,|\,1 \le i\le n\}$. For
any $i< j$, set $\Sigma(v_{i,j})=\{v_i,v_j\}$, and set
$\Sigma(v)=\emptyset$ for all other vertices.  Note that $\Sigma$ is
a $2$-neighborhood system and $\rho(\Sigma)=2$, but $\chi(\Sigma)\ge n$. Hence,
the family $\MF=\{S_n\,|\,n\ge 1\}$ is not even $\sigma$-bounded at
depth 2 (while it clearly has bounded expansion).

\smallskip

A way to circumvent this would be to add a notion of complexity to
that of a neighborhood system $\Sigma$. A \emph{realizer} $R$ for a $d$-neighborhood system $\Sigma$
is a set of paths of length at most $2d$ between all pairs $u,v$
such that $u \in \Sigma(v)$. Given such a realizer $R$, we define
$\lambda(R)$ as the maximum over all vertices $u$ of $G$ of the number
of paths of $R$ containing $u$. The \emph{complexity}
$\lambda(G,\Sigma)$ is the minimum of $\lambda(R)$ over all realizers
of $\Sigma$. 

If we only allow $d$-neighborhood systems $\Sigma$ with
$\lambda(G,\Sigma)$ bounded by a fixed function of $d$, then it can be
shown that having bounded expansion is equivalent to being
$\sigma$-bounded at each depth. We omit the details, since the proof
is very close from that showing that the immersion and topological
minor resolutions are equivalent (see~\cite[Section 5.8]{NO}).

\medskip

Ne\v set\v ril and Ossona de Mendez~\cite{NO11} call a class $\MF$ of
graphs \emph{somewhere dense} if there exists an integer $d$, such
that the set of shallow topological minors at depth $d$ of graphs of
$\MF$ is the set of all graphs. Otherwise $\MF$ is \emph{nowhere
  dense}.

\smallskip

If a graph $G$ contains a $(\le\!2d-1)$-subdivision of a clique on $k$
vertices, then as above we can find a $d$-neighborhood system $\Sigma$ for $G$
with $\rho(\Sigma)=2$, and
$\omega(\Sigma)\ge k$. This remark has the following direct
consequence.

\begin{lemma}\label{lem:be2}
  If there is a function $f$, such that for any $d\ge 1$, for any $G
  \in \MF$ and for any $d$-neighborhood system $\Sigma$ for $G$,
  we have $\omega(\Sigma)\le f(\rho(\Sigma),d)$, then
  $\MF$ is nowhere-dense.
\end{lemma}

Again, we note that if we require that the complexity of the
neighborhood systems $\Sigma$ we consider is bounded by a function of the depth,
the two properties in Lemma~\ref{lem:be2} are indeed equivalent.

\section{Remarks}

If we add to the definition of $\Sigma$-coloring the condition that
the coloring must be a proper coloring of $G$ (which is the way it was
defined in~\cite{AEH08}), then the same proofs work. This shows that
the two definitions of $\Sigma$-coloring are equivalent: a class is
$\sigma$-bounded according to one definition if and only if it is
$\sigma$-bounded according to the other (and the same holds for the
list versions).

\medskip

A class of graphs for which the chromatic number is bounded by a
function of the clique number is said to be \emph{$\chi$-bounded}. Such
classes have been extensively studied and many deep conjectures remain
open about structural aspects of graphs in these classes. It is
interesting to note that in our setting, classes of graphs for which
$\chi(\Sigma)$ is bounded by a function of $\omega(\Sigma)$, are
much easier to understand: they correspond precisely to
$\sigma$-bounded classes. Indeed if a class is $\sigma$-bounded, then
$\chi(\Sigma)$ is clearly bounded by a function of
$\omega(\Sigma)$. Conversely if in $\MF$, $\chi(\Sigma)$ is
bounded by a function of $\omega(\Sigma)$, then $\MF$ cannot contain
(as subgraphs of graphs in the class) 1-subdivisions of arbitrarily
large cliques (if it did, we could also find as subgraphs of graphs in
$\MF$ 1-subdivisions of graphs from a class that is not
$\chi$-bounded). Then by Theorem~\ref{th:3}, $\omega(\Sigma)$ is
bounded by a function of $\rho(\Sigma)$ in $\MF$, and consequently
$\chi(\Sigma)$ is also bounded by a function of
$\rho(\Sigma)$. It follows that $\MF$ is $\sigma$-bounded.

\paragraph{Acknowledgement} The two authors are grateful to Omid
Amini for his kind suggestions and remarks.


\end{document}